\documentclass[a4paper,twoside,12pt]{article}

\usepackage{amssymb,amsmath,amsthm,latexsym}
\usepackage{amsfonts}
\usepackage{graphicx,caption,subcaption}
\usepackage[pdftex,bookmarks,colorlinks=false]{hyperref}
\usepackage[hmargin=1.2in,vmargin=1.2in]{geometry}
\usepackage[mathscr]{euscript}
\DeclareMathOperator{\lcm}{lcm}

\usepackage[auth-lg,affil-sl]{authblk}
\allowdisplaybreaks

\usepackage{float}
\usepackage{calc,tikz}
\usetikzlibrary{arrows,decorations.markings,positioning,shapes.geometric,}
\tikzstyle{vertex}=[circle, draw, inner sep=1pt, minimum size=4pt]
\newcommand{\vertex}{\node[vertex]}

\tikzstyle{ann} = [fill=white,font=\footnotesize,inner sep=1pt]
\tikzstyle{arrow} = [thick,<-->,>=stealth]

\usepackage{multirow}
\usepackage{hhline}
\usepackage{booktabs}
\usepackage{longtable}

\usepackage{url}

\newcommand{\noi}{\noindent}
\newcommand{\N}{\mathbb{N}}

\newtheorem{theorem}{Theorem}[section]
\newtheorem{definition}[theorem]{Definition}
\newtheorem{lemma}[theorem]{Lemma}
\newtheorem{proposition}[theorem]{Proposition}
\newtheorem{corollary}[theorem]{Corollary}

\newtheorem{observation}[theorem]{Observation}

\newtheorem{example}{Example}

\newcommand{\keywordsname}{Keywords}

\newcommand{\mscname}{MSC 2010}

\allowdisplaybreaks

\title{\textbf{\sc On Derivative Euler Phi Function Set-Graphs}}

\author{Johan Kok$^\ast$}
\affil{\small Centre for Studies in Discrete Mathematics\\ Vidya Academy of Science \& Technology \\Thalakkottukara, Thrissur, Kerala, India.\\{\tt kokkiek2@tshwane.gov.za,}}

\author{Eunice Gogo Mphako-Banda}
\affil{\small School of Mathematical Sciences\\ University of Witswatersrand \\Johannesburg, South Africa.\\{\tt eunice.mphako-banda@wits.ac.za}}

\author{Sudev Naduvath\footnote{Corresponding author}}
\affil{\small Department of Mathematics\\ CHRIST (Deemed to be University) \\Thalakkottukara, Thrissur, Kerala, India.\\ {\tt sudevnk@gmail.com}}

\date{}

\begin{document}
\maketitle

\hrule

\begin{abstract}
In this paper, we study some graph theoretical properties of two derivative Euler Phi function set-graphs. For the Euler Phi function $\phi(n)$, $n\in \N$, the set $S_\phi(n) =\{i:\gcd(i,n)=1, 1\leq i \leq n\}$ and the vertex set is $\{v_i:i\in S_\phi(n)\}$. Two graphs $G_d(S_\phi(n))$ and $G_p(S_\phi(n))$, defined with respect to divisibility adjacency and relatively prime adjacency conditions, are studied. 
\end{abstract}

\noi\textbf{Keywords:} Euler Phi set-graph, least common multiple set-graph, relatively prime set-graph.

\vspace{0.25cm}

\noi \textbf{Mathematics Subject Classification 2010:} 05C15, 05C38, 05C75, 05C85. 
\vspace{0.25cm}

\hrule

\section{Introduction}

For all  terms and definitions, not defined specifically in this paper, we refer to \cite{BM1,BLS,FH,DBW}. Unless mentioned otherwise, all graphs considered here are undirected, simple, finite and connected.

The \textit{order} and \textit{size} of a graph $G$ are respectively the numbers of vertices and edges in it and are denoted  respectively by $nu(G)$ and $\varepsilon(G)$. In this paper, we write $\nu(G) = n \geq 1$ and size $\varepsilon(G)= p \geq 0$. The minimum and maximum degrees of $G$ are represented by $\delta(G)$ and $\Delta(G)$, respectively. When the context is clear, we write these notations simply as $\delta$ and $\Delta$ respectively. The degree of a vertex $v \in V(G)$ is denoted $d_G(v)$ or when the context is clear, simply as $d(v)$. The complement graph of graph $G$ is denoted by, $\overline{G}$ and for a graph of composite notation such as $G_{g(x)}(f(x))$ the complement graph will be denoted by, $\overline{G}_{g(x)}(f(x))$. Unless mentioned otherwise all graphs $G$ are simple, connected and finite graphs.

The sequences of intergers, $a_i,\ 1\leq i\leq k\ a_i\in \{1,2,3\dots,n\}$ of greatest length such that either $a_j\mid a_{j+1}$ or $a_{j+1}\mid a_j$, $1\leq j\leq k-1$ have been studied in \cite{EFH} and the longest sequences are equivalent to finding a longest path in a divisor graph as defined in \cite{CMSZ}. 

The \textit{divisor graph} of a finite, non-empty set $S$ of positive integers, denoted by $G_d(S)$, is the graph with has vertex set $V(G_d(S)) = \{v_i:i\in S\}$ and edge set, $E(G_d(S)) = \{v_iv_j: i|j,\ \text{or}\ j|i\ \ \text{and}\; i\neq j\}$. Hence, if $S=\{1,2,3,\dots,n\}$, then finding a longest path in $G(S)$ is equivalent to the subject of study in [3]. The importance of divisor and relatively prime graphs is founded in some of the applications in communication networks.

Recall that the \textit{relatively prime graph} of a finite, non-empty set of positive integers $S$ is the graph $G_p(S)$ with vertex set $V(G_p(S)) = \{v_i:i\in S\}$ and edge set $E(G_p(S)) = \{v_iv_j: \gcd(i,j)=1\}$ (see \cite{CMSZ}). Note that $1$ is relatively prime to all positive integers and any integer is divisible by $1$. On the other hand, a positive integer $a\geq 2$ is not relatively prime to itself. Therefore, no loops can be found in $G_p(S)$. A comprehensive study on the classification of divisor graphs which includes interesting results related to the complement of graphs and induced neighbourhoods, has been done in \cite{CF1}. An interesting conjecture that for $S=\{1,2,3,\ldots,n\}$, every tree of order $n$ is a subgraph of the relatively prime graph, $G_p(S)$ was verified for $n\leq 15$ in \cite{FH1}. A specialised \textit{relatively prime graph} has been defined in \cite{SS1}. For a given integer $n\in \N$, the \textit{Euler Phi graph}, denoted by $G_p(S_\phi(n))$, is a relatively prime graph on the set of positive integers $S_\phi(n) = \{i: \gcd(i, n)=1, 1\leq i\leq n\}$ with vertex set $V(G_p(S_\phi(n)))=\{v_i:i\in S_\phi(n)\}$. Note that the vertices are not necessarily consecutively indexed. The notation used in this paper differs slightly from that introduced in \cite{SS1} because for $\phi(n)\in \N$ and it is possible to have that, $n_1\neq n_2$ and $\phi(n_1)=\phi(n_2)$.


\section{New Directions}

For the set $S_\phi(n)$, two graphs with the context of this study is possible -- a relatively prime graph as studied in \cite{SS1} and a divisor graph. Let $S'_\phi(n)=S_\phi(n)-\{v_1\}$. If for all pairs of distinct vertices $v_i,v_j \in S'(\phi)$ the integers $i,j$ are either divisible or relatively prime, then it follows that the divisor graph $G_d(S_\phi(n))= \overline{G}_p(S_\phi(n)-\{v_1\}) + v_1$ and the relatively prime graph, $G_p(S_\phi(n))= \overline{G}_d(S_\phi(n)-\{v_1\}) + v_1$. It is clear that for both the aforesaid graphs, $\Delta = \phi(n)-1$.

Let $\mathcal{P}$ be the set of prime numbers. Also, let $\mathcal{P}(k)$ denote the first $k$ consecutive prime numbers. We know that if a positive integer $n$ is written as some product of $k$ prime numbers, $n= p_i^{a_1}\cdot p_j^{a_2}\cdots p_t^{a_k}$, then $\phi(n)=n(1-\frac{i}{p_i})(1-\frac{1}{p_j})\cdots (1-\frac{1}{p_t})$ (see \cite{TMA}). This result will be used throughout with little reference to other identies of the $\phi$-function. This result provides the cardinality of $S_\phi(n)$.

\begin{theorem}\label{Thm-2.1}
A divisor graph $G_d(S_\phi(n))$ is acyclic if and only if the corresponding relatively prime graph $G_p(S_\phi(n))$ is complete. 
\end{theorem}
\begin{proof}
Let a relatively prime graph $G_p(S_\phi(n))$ be complete. Hence, besides the vertex $v_1$, all other pairs of distinct vertices are relatively prime so they are not divisors. Therefore, in the corresponding divisor graph only $v_1$ is adjacent to all vertices which renders a star graph. So, $G_d(S_\phi(n))$ is acyclic.

Let a divisor graph, $G_d(S_\phi(n))$ be acyclic. The converse result follows by similar reasoning.
\end{proof}

We recall the next well known lemma \cite{ES1}.

\begin{lemma}{\rm \cite{ES1}}\label{Lem-2.2}
\begin{enumerate}
\item[(i)] Any positive, non-prime integer, $n\geq 4$, is a multiple of at least one prime number, and all of which are less than $n$.\\
\item[(ii)] $\prod\limits_{i=1}^{k}p_i < p^2_{k+1}$ for $p_i \in \mathcal{P}(k)$.
\end{enumerate}
\end{lemma}

For a positive integer $n$ let the prime divisor (or factor) set of $n$ be, $S_\vartheta(n) =\{p_i:p_i| n\}$.

\begin{theorem}\label{Thm-2.3}
For any positive integer $n\ge 5$, the divisor graph $G_d(S_\phi(n))$ is acyclic if and only if $S_\vartheta(n)= \mathcal{P}(k)$, for some $k\in \N$. 
\end{theorem}
\begin{proof}
If $S_\vartheta(n)=\mathcal{P}(k)$ for some $k\in \N$ then, $n=t\cdot \prod\limits_{i=1}^{k}p_i$, $t\in \N$. Two cases must be considered.

\textit{Case 1}: For $t=1$ and for all non-prime $m, 1\leq m\leq p_k+1$, $\gcd(m, n)\neq 1$ because from Lemma \ref{Lem-2.2}(i) it follows that, $m$ is a product of prime numbers in $\mathcal{P}(k)$. Similarly, $\forall~t$, $p_k+1< t \leq n$ and non-prime, we have $t \notin S_\phi(n)$ because from Lemma \ref{Lem-2.1}(ii), $p^2_{k+1}>n$ and hence $\gcd(t,n)\neq 1$. Therefore, the set $S_\phi(n) = \{1\}\cup \{p_j: (k+1)\leq j\leq \ell,~p_\ell <n\}$. Since all pairs of distinct prime numbers are non-divisors to each other it follows that $G_d(S_\phi(n))$ is the star graph $S_{1,(\ell -k)}$. 

Now, if $G_d(S_\phi(n))$ is acyclic and since $1| z,\ \forall\, z\in \N$, it follows that $G_d(S_\phi(n)\cong S_{1,(\phi(n)-1)}$. Further for all pairs of distinct $a,b \in S_\phi(n), a\neq 1, b\neq 1$, it follows that, $\gcd(a,b)= 1$. We prove two subcases.

\textit{Subcase 1(a)}: If $n$ is odd. For $n \geq5$ it follows easily that $2,4 \in S_\phi(n)$ hence, $G_d(S_\phi(n))$ is cyclic. Also $S_\vartheta(n)\neq \mathcal{P}(k)$ for any $k$ because $2\notin S_\vartheta(n)$. Hence, $n$ cannot be odd.

\textit{Subcase 1(b)}: If $n\geq 4$ and even, no product of a combination of elements in $S_\vartheta(n)$ or a multiple thereof can be an element in $S_\phi(n)$ else, a triangle exists in $G_d(S_\phi(n))$. Therefore, all elements in $S_\phi(n)$ are prime numbers, $p_{k+1}\leq p_j < n$ for some $k\in \N$. If not, a triangle must exist which is a contradiction. The aforesaid implies that, $S_\vartheta(n)=\mathcal{P}(k)$ for some $k\in \N$.

\textit{Case 2:} For $t\geq 2$ the result follows through immediate induction by similar reasoning found in Case-1.
\end{proof}

The lower bound on $n$ in Theorem \ref{Thm-2.3} follows from the fact that for $1\leq n\leq 4$ it is easy to verify that $G_d(S_\phi(n))$ is acyclic. Also, $n=3$ is the only exception to Subcase i(a).

\begin{example}{\rm 
When $n=24$, the prime divisor set $S_\vartheta(24)=\{2,3\}= \{2,3\}=\mathcal{P}(2)$, the divisor graph $G_d(S_\phi(24))$ is cyclic. Clearly, $S_\phi(24)=\{1,5,7,11,13,17,19,23\}$. See Figure \ref{fig:fig-1}. 

\begin{figure}[h!]
\centering
\begin{tikzpicture}[auto,node distance=2.5cm, thick,main node/.style={circle,draw,font=\sffamily\Large\bfseries}]
\vertex [circle,fill] (1) at (0:0) [label={left:$v_1$}]{};
\vertex [circle,fill] (2) at (60:3) [label={above:$v_5$}]{};
\vertex [circle,fill] (3) at (40:4) [label={right:$v_7$}]{};
\vertex [circle,fill] (4) at (20:5) [label={right:$v_{11}$}]{};
\vertex [circle,fill] (5) at (0:5.5) [label={right:$v_{13}$}]{};
\vertex [circle,fill] (6) at (340:5) [label={right:$v_{17}$}]{};
\vertex [circle,fill] (7) at (320:4) [label={right:$v_{19}$}]{};
\vertex [circle,fill] (8) at (300:3) [label={right:$v_{23}$}]{};
\path[every node/.style={font=\sffamily\small},line width=0.75pt]
(1) edge (2)
(1) edge (3)
(1) edge (4)
(1) edge (5)
(1) edge (6)
(1) edge (7)
(1) edge (8)
;
\end{tikzpicture}
\caption{}\label{fig:fig-1}
\end{figure}

Since, for $n=15$ the prime divisor set, $S_\vartheta(9)=\{3,5\}\neq \{2,3,5\}=\mathcal{P}(3)$, the divisor graph, $G_d(S_\phi(15))$ is cyclic. $S_\phi(15) = \{ 1,2,4,6,7,8,11,13,14\}$. See Figure \ref{fig:fig-2}.

\begin{figure}[h!]
\centering
\begin{tikzpicture}[auto,node distance=2.5cm, thick,main node/.style={circle,draw,font=\sffamily\Large\bfseries}]
\vertex [circle,fill] (1) at (0:0) [label={left:$v_1$}]{};
\vertex [circle,fill] (2) at (55:3) [label={above:$v_2$}]{};
\vertex [circle,fill] (3) at (40:4.5) [label={right:$v_4$}]{};
\vertex [circle,fill] (4) at (25:5) [label={right:$v_6$}]{};
\vertex [circle,fill] (5) at (10:5.5) [label={right:$v_7$}]{};
\vertex [circle,fill] (6) at (355:5.5) [label={right:$v_8$}]{};
\vertex [circle,fill] (7) at (340:5) [label={right:$v_{11}$}]{};
\vertex [circle,fill] (8) at (325:4.5) [label={right:$v_{13}$}]{};
\vertex [circle,fill] (9) at (310:3) [label={below:$v_{14}$}]{};
\path[every node/.style={font=\sffamily\small},line width=0.75pt]
(1) edge (2)
(1) edge (3)
(1) edge (4)
(1) edge (5)
(1) edge (6)
(1) edge (7)
(1) edge (8)
(1) edge (9)
(2) edge (3)
(2) edge (4)
(2) edge (6)
(2) edge (9)
(3) edge (6)
(5) edge (9)
;
\end{tikzpicture}
\caption{}\label{fig:fig-2}
\end{figure}

}\end{example}

The following corollary is a direct consequence of Theorem \ref{Thm-2.3}.

\begin{corollary}
If $G_d(S_\phi(n))$ is acyclic, then $n$ is even.
\end{corollary}

\begin{theorem}
For $n\in \N$, let $x= \min\{y:y\in S_\phi(n)-\{1\}\}$. The divisor graph $G_d(S_\phi(n))$ is acyclic if and only if $n\leq x^2$.
\end{theorem}
\begin{proof}
For $n\in \N$ let $x= \min\{y:y\in S_\phi(n)-\{1\}\}$. From Corollay 1.4 it follows that if suffices to only consider $n$ is even. Therefore, $x\geq 3$ is a prime number. Also since, all elements of $S_\phi(n)-\{1\}$ are either prime numbers or multiples of prime numbers, it follows that if $n\leq x^2$. Hence, all elements of $S_d(\phi(n))-\{1\}$ are prime. Therefore, $G_d(S_\phi(n))$ is acyclic.

Next, if $G_d(S_\phi(n))$ is acyclic, the converse is obvious.
\end{proof}

The next corollary is a useful summary of equivalent results.

\begin{corollary}
For any $n \in \N$, $n \geq 5$ it follows that $S_\vartheta(n)= \mathcal{P}(k)$, for some $k\in \N$ if and only if:
\begin{enumerate}\itemsep0mm
\item[(i)] $G_d(S_\phi(n))$ is acyclic.
\item[(ii)] $G_p(S_\phi(N))$ is complete.
\item[(iii)] $n\leq x^2$ for $x= \min\{y:y\in S_\phi(n)-\{1\}\}$.
\end{enumerate}
\end{corollary}

\begin{proposition}
For $n\geq 5$ and $S_\vartheta(n) = \mathcal{P}(k)$ for some $k\in \N$, then
\begin{enumerate}\itemsep0mm
\item[(i)] $\varepsilon(G_d(S_\phi(n)))=|\{p_i:(k+1)\leq i\leq \ell, p_\ell<n\}|=\Delta$.
\item[(ii)] $\varepsilon(G_p(S_\phi(n)))=\frac{1}{2}\Delta (\Delta +1)$.
\end{enumerate}
\end{proposition}
\begin{proof}
\item[(i)] Since $G_d(S_\phi(n))$ is the star graph, $S_{1,(\ell -k)}$, by the proof of Theorem \ref{Thm-2.3} the result is obvious.
\item[(ii)] Since $G_p(S_\phi(n))$ is of order $\Delta +1$ and is a complete graph, the result is obvious.
\end{proof}

\begin{observation}{\rm A number of trivial observations on these invariants follow immediately.
\begin{enumerate}\itemsep0mm 
\item[(i)] The respective domination numbers are $1$. 
\item[(ii)] The chromatic numbers $\chi(G_d(S_\phi(n)))=2$ and $\chi(G_p(S_\phi(n)))=\Delta+1$. 
\item[(iii)] The clique numbers $\omega(G_d(S_\phi(n)))=2$ and $\omega(G_p(S_\phi(n)))=\Delta+1$. 
\item[(iv)] These graphs $G_d(S_\phi(n))$ and $G_p(S_\phi(n))$ are perfect graphs.
\end{enumerate}
}\end{observation}

\section{Derivative Set-graphs}

The notion of a set-graph was introduced in \cite{KCSS} as explained below.

\begin{definition}\label{Defn-3.1}{\rm \cite{KCSS} 
Let $A^{(n)} = \{a_1,a_2,a_3,\ldots,a_n\}$, $n\in \N$ be a non-empty set and the $i$-th $s$-element subset of $A^{(n)}$ be denoted by $A^{(n)}_{s,i}$. Now, consider $\mathcal S = \{A^{(n)}_{s,i}: A^{(n)}_{s,i} \subseteq A^{(n)}, A^{(n)}_{s,i} \neq \emptyset \}$. The \textit{set-graph} corresponding to set $A^{(n)}$, denoted by $G_{A^{(n)}}$, is defined to be the graph with $V(G_{A^{(n)}}) = \{v_{s,i}: A^{(n)}_{s,i} \in \mathcal S\}$ and $E(G_{A^{(n)}}) = \{v_{s,i}v_{t,j}: A^{(n)}_{s,i} \cap A^{(n)}_{t,j} \neq \emptyset\}$, where $s\neq t$ or $i\neq j$.
}\end{definition}

The largest complete subgraph in the given set-graph $G_{A^{(n)}}$, $n \geq 2$  is $K_{2^{n-1}}$ and the number of such largest complete subgraphs in the given set-graph $G_{A^{(n)}}$, $n \geq2$ is provided in the following proposition.

\begin{proposition}\label{Prop-3.2}{\rm \cite{KCSS}}
The set-graph $G_{A^{(n)}}, n\geq 1$ has exactly $2^{n-1}$ largest complete subgraphs $K_{2^{n-1}}$.
\end{proposition}

The graph in Figure \ref{fig:fig-3} is the set-graph $G_{A^{(3)}}$.

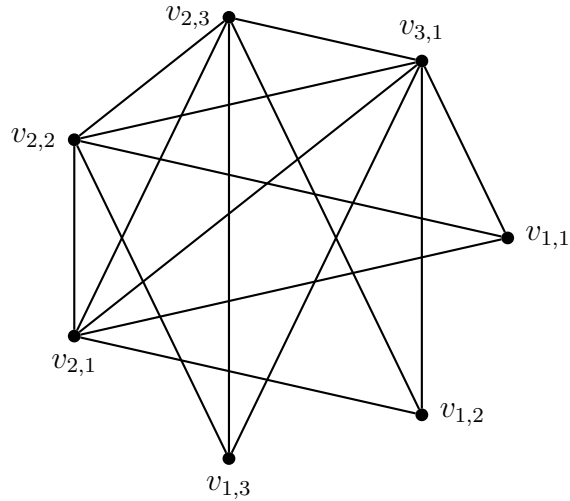
\begin{figure}[h!]
\centering
\begin{tikzpicture}[auto,node distance=1.75cm,
thick,main node/.style={circle,draw,font=\sffamily\Large\bfseries}]
\vertex (v1) at (0:3) [fill,label=right:$v_{1,1}$]{};
\vertex (v2) at (308.57:3) [fill,label=right:$v_{1,2}$]{};
\vertex (v3) at (257.14:3) [fill,label=below:$v_{1,3}$]{};
\vertex (v4) at (205.71:3) [fill,label=below:$v_{2,1}$]{};
\vertex (v5) at (154.28:3) [fill,label=left:$v_{2,2}$]{};
\vertex (v6) at (102.85:3) [fill,label=left:$v_{2,3}$]{};
\vertex (v7) at (51.42:3) [fill,label=above:$v_{3,1}$]{};
\path 
(v1) edge (v4)
(v1) edge (v5)
(v1) edge (v7)
(v2) edge (v4)
(v2) edge (v6)
(v2) edge (v7)
(v3) edge (v5)
(v3) edge (v6)
(v3) edge (v7)
(v4) edge (v5)
(v4) edge (v6)
(v4) edge (v7)
(v5) edge (v6)
(v5) edge (v7)
(v6) edge (v7)
;
\end{tikzpicture}
\caption{\small The set-graph $G_{A^{(3)}}$.}\label{fig:fig-3}
\end{figure}

\begin{example}{\rm 
For $A^{(4)} = \{1,3,5,7\}$. Following from Definition 2.1 set-graph has vertices, $v_{1,1} =\{1\}$, $v_{1,2}=\{3\}$, $v_{1,3}=\{5\}$, $v_{1,4}=\{7\}$, $v_{2,1}=\{1,3\}$, $v_{2,2}=\{1,5\}$, $v_{2,3}=\{1,7\}$, $v_{2,4}=\{3,5\}$, $v_{2,5}=\{3,7\}$, $v_{2,6}=\{5,7\}$, $v_{3,1}=\{1,3,5\}$, $v_{3,2}=\{1,3,7\}$, $v_{3,3}=\{1,5,7\}$, $v_{3,4}=\{3,5,7\}$, $v_{4,1}=\{1,3,5,7\}$. The $\iota$-weights are given by the mapping, $\iota(v_{1,1}) \mapsto 1$, $\iota(v_{1,2})\mapsto 3$, $\iota(v_{1,3})\mapsto 5$, $\iota(v_{1,4})\mapsto 7$, $\iota(v_{2,1})\mapsto 3$, $\iota(v_{2,2})\mapsto 5$, $\iota(v_{2,3})\mapsto 7$, $\iota(v_{2,4})\mapsto 15$, $\iota(v_{2,5})\mapsto 21$, $\iota(v_{2,6})\mapsto 35$, $\iota(v_{3,1})\mapsto 15$, $\iota(v_{3,2})\mapsto 21$, $\iota(v_{3,3})\mapsto 35$, $\iota(v_{3,4})\mapsto 105$, $\iota(v_{4,1})\mapsto 105$. 

Hence, $\iota(A^{(4)}) = \{v_{1,1}(1), v_{1,2}(3), v_{1,3}(5), v_{1,4}(7), v_{2,1}(3),v_{2,2}(5), v_{2,3}(7), v_{2,4}(15),\\ v_{2,5}(21), v_{2,6}(35),v_{3,1}(15), v_{3,2}(21), v_{3,3}(35),v_{3,4}(105), v_{4,1}(105)\}$. 
}\end{example}

\begin{proposition}\label{Prop-3.3}
For a set of positive integers $S =\{a_1,a_2,a_3,\ldots,a_n\}$ with $a_1 < a_2 < a_3 <\cdots < a_n$, we have
\begin{enumerate}\itemsep0mm
\item[(i)] If $a_1\neq 1$  and all pairs of distinct entries are relatively prime then the $\iota$-weights are distinct.
\item[(ii)] If $a_1=1$ and all pairs of distinct entries are relatively prime then the $\iota$-weights repeat twice except for $v_{1,1}$.
\end{enumerate}
\end{proposition}
\begin{proof}
\item[(i)] The result follows immediately from the facts that all subsets are distinct and for any set of relative prime integers $X=\{e_1,e_2,e_3,\ldots,e_t\}$ the least common multiple, $\lcm(e_i\in X)=\prod\limits_{i=1}^{t}e_i$ for all $e_i$.
\item[(ii)] The result follows immediately from the facts that, $1\cdot x=x$ and for each subset $X$, $1\notin X$ there exists a subset $\{1\}\cup X$ except for the unique vertex $v_{1,1}=\{1\}$. Hence, $\iota(v_{1,1})=1$, uniquely. All other $\iota$-weights repeat twice.
\end{proof}

In the context of repetition we say that an $\iota$-weight which is unique, repeats zero times.

\begin{theorem}\label{Thm-3.4}
For a set of positive integers $S =\{a_1,a_2,a_3,\ldots,a_n\}$ with $a_1 < a_2 < a_3 <\cdots < a_n$, all $\iota$-weights repeat an even number of times.
\end{theorem}
\begin{proof}
We prove the result by mathematical induction on the cardinality of the set $S$. For $S=\{a_1\}$, the corresponding set-graph has the isolated vertex $v_{1,1}$. Therefore, $\iota(v_{1,1})= a_1$. Since the $\iota$-weight repeats zero times, the result holds true when $|S|=1$. For $S=\{a_1,a_2\}$, the corresponding set-graph has vertex set $\{v_{1,1}, v_{1,2}, v_{2,1}\}$ and hence $\iota(V_{1,1})= a_1$, $\iota(V_{1,2})=a_2$ and $\iota(v_{2,1}) = a_1\cdot a_2$. If $a_1=1$, then $\iota(v_{1,2})=\iota(v_{2,1})=a_2$. Hence, the $\iota$-weight $a_1$ repeats zero times and the $\iota$-weight $a_2$ repeats twice. If $a_1\neq 1$ then, $\iota(v_{1,1})=a_1$, $\iota(v_{1,2})=a_2$ and $\iota(v_{2,1})= a_1\cdot a_2$. Clearly, each $\iota$-weight repeats zero times and the result is true for $|S|=2$. 

Now, assume that the result holds for any set $S$ of cardinality $k$. Consider a set of positive integers $S'= \{e_1,e_2,e_3,\ldots,e_k,e_{k+1}\}$. Then for $S= \{e_1,e_2,e_3,\ldots,e_k\}$, the result holds. The vertices of set-graph $G_{S'^{(k+1)}}$ can be partitioned into three subsets namely, $V(G_{S^{(k)}})$, $\{\{e_{k+1}\}\cup v_{s,i}: v_{s,i}\in V(G_{S^{(k)}})\}$ and $\{v_{k+1,1}\}$. Since $k+1 >1$ the $\iota$-weights of the vertices in $\{\{e_{k+1}\}\cup v_{s,i}: v_{s,i}\in V(G_{S^{(k)}})\}$ repeat as $(k+1)$-multiples the equal even number of times as the corresponding vertices in $V(G_{S^{(k)}})$. Finally, $\iota(v_{k+1,1})$ is unique and hence it repeats zero times. In conclusion the results holds for any set of positive integers of cardinality $k+1$. Thus, the result holds for all finite sets of positive integers of cardinality $n\in \N$ by mathematical induction.
\end{proof}

The \textit{first derivative set-graph} is a divisor graph with respect to the LCM values. Hence, vertices $v_{s,i}(j)$ and $v_{t,k}(\ell)$ are adjacent if and only if $j\mid \ell$ or $\ell \mid j$. The \textit{second derivative set-graph} is a relatively prime graph with respect to the LCM values. Hence, vertices $v_{s,i}(j)$ and $v_{t,k}(\ell)$ are adjacent if and only if $\gcd(j,\ell)=1$. These derivative set-graphs will be investigated for the Euler Phi function $\phi(n)$.

\subsection{Euler Phi Set-graphs and derivative set-graphs}

In order to link some aspects of graph theory with the notion of divisor graph, relative prime graphs, Euler Phi set-graphs and alike, we introduce the notion of a set which divides another set. If for two distinct non-empty sets $X$ and $Y$, we have $X\cap Y\neq \emptyset$, we say $X$ divides $Y$ and write it as $X\mid Y$. Clearly it is true that, $X\mid Y\Leftrightarrow Y\mid X$. Also, it is not necessarily true that, $(X\mid Y, Y\mid Z)\Rightarrow X\mid Z$.

The \textit{Euler Phi set-graph} denoted by $G_{S_\phi(n)}$ is obtained by applying Definition \ref{Defn-3.1} to the set $S_\phi(n)$. Clearly, $G_{S_\phi(n)}$ has all the properties of $G_{A^{(\phi(n))}}$ for $A^{(\phi(n))} = \{a_1,a_2,a_3,\ldots,a_{\phi(n)}\}$. Then, we have a straightforward corollary.

\begin{corollary}\label{Cor-3.5}
$G_{S_\phi(n)} \cong G_{S_\phi(m)}$ if and only if $\phi(n)=\phi(m)$.
\end{corollary}

We denote the Euler Phi $\lcm$-divisor set-graph by, $G_d(\iota(S_\phi(n)))$. Recall that except for $n=1,2$, $\phi(n)$ is even for all $n\in \N/\{1,2\}$. For $n= 1,2$, $\phi(n) = 1$ and for $n=3,4,6$, $\phi(n)=2$ and for $n=5,8,10,12$, $\phi(n) = 4$. Generally the vertex notation $v_{s,i}(\iota)$ will mean the vertex corresponding to the $i^{th}$, $s$-element subset with corresponding $\iota$-weight,

\begin{lemma}\label{Lem-3.6}
An Euler Phi $\lcm$-divisor set-graph $G_d(\iota(S_\phi(n)))$ has exactly three vertices with maximum degree, $\Delta (G_d(\iota(S_\phi(n))))= 2^{\phi(n))}-2$.
\end{lemma}
\begin{proof}
Since $\iota(v_{1,1})=1$, the vertex $v_{1,1}(1)$ is adjacent to all vertices in the Euler Phi $\lcm$-divisor set-graph. Also, $\iota(v_{\phi(n)-1,\phi(n)}) =\iota(v_{\phi(n),1})=lcm(S_\phi(n)) =max\{\iota$-$weight:over~all~vertices\}$ and $\iota(v_{s,i})\mid lcm(S_\phi(n))$ for all vertices in the Euler Phi $\lcm$-divisor set-graph. Hence, all vertices are adjacent to vertices $v_{\phi(n)-1,\phi(n)}(\iota)$, $v_{\phi(n),1}(\iota)$, respectively. Therefore, the result.
\end{proof}

The following theorem describes the vertex degrees of Euler Phi $\lcm$-divisor set-graph.

\begin{theorem}\label{Thm-3.7}
All vertex degrees of an Euler Phi $\lcm$-divisor set-graph $G_d(\iota(S_\phi(n)))$ are even. 
\end{theorem}
\begin{proof}
Consider a Euler Phi prime set $\{1,p_1,p_2,\ldots,p_t\}$. In the set-graph which has an odd number that is, $2^{t+1}-1$ vertices, the vertex $v_{1,1}(1)$ is adjacent to all vertices thus has even degree. Without loss of generality consider any 1-element subset $\{p_i\}$. The element $p_i$ is an element in exactly $2^t$ subsets hence, vertex $v_{1,i+1}(p_i)$ will be adjacent to the remaining odd number of vertices corresponding to those subsets. Since, vertex $v_{1,i+1}(p_i)$ is adjacent to vertex $v_{1,1}(1)$ it has even degree in respect of the vertices considered thus far. Furthermore, if $\iota(v_{1,i+1})=p_i$, is a divisor of $\iota(v_{s,k}(j))$ which from Theorem \ref{Thm-3.4} repeats even times, a further even adjacency number is added to the degree of $v_{1,i+1}(p_i)$. Clearly after exhaustive adjacency count the degree of $v_{1,i+1}(p_i)$ is even. Similar reasoning holds for any subset hence, for any corresponding vertex in the graph $G_d(\iota(S_\phi(n)))$. Therefore the result.  
\end{proof}

\noi In view of the above results, we have some interesting corollaries as given below.

\begin{corollary}
An Euler Phi $\lcm$-divisor set-graph is an Eulerian graph.
\end{corollary}
\begin{proof}
Since the vertex $v_{1,1}(1)$ is adjacent to all vertices in an Euler Phi $\lcm$-divisor set-graph, the graph is connected and obviously, non-empty. Since it is well-known (see \cite{BM1}), that a graph is eulerian if and only if it is non-empty and connected and it has no vertices of odd degree, Theorem \ref{Thm-3.7} implies that an Euler Phi $\lcm$-divisor set-graph is Eulerian.
\end{proof}

\begin{corollary}
$\chi(G_{S_\phi(n)})\leq \chi(G_d(\iota(S_\phi(n))))$, where $\chi$ denotes the chromatic number of a graph.
\end{corollary}
\begin{proof}
Since, $G_d(\iota(S_\phi(n)))$ is a proper super graph of the corresponding set-graph, $G_{S_\phi(n)}$ the result follows immediately.
\end{proof}

\begin{corollary}
$\chi(G_d(\iota(S_\phi(m)))) =\omega(G_d(\iota(S_\phi(m))))$ for $m\in \N$, where $\omega$ is the clique number of a graph.
\end{corollary}
\begin{proof}
It is known that the relation between the clique number and the chromatic number of a graph is given by, $\omega(G)\leq \chi(G)$. It is also known that if $H\subseteq G$ then, $\chi(H)\leq \chi(G)$. Furthermore, if $p=n+1$ is a prime number then, $\phi((p)= n$ and $S_\phi(p) =\{1,2,3,\ldots,n\}$.  From Theorem 2 in \cite{SS1}, it follows that $\chi(G_d(\iota(S_\phi(p)))) = \lfloor log_2(2^n -1)\rfloor +1 =\omega(G_d(\iota(S_\phi(p))))$. In \cite{YT1}, it is shown that $G_d(\iota(S_\phi(p)))$ is a perfect graph. Therefore, for any $m$ for which $S_\phi(m)\subseteq S_\phi(p))$ the Euler Phi $\lcm$-divisor set-graph $G_d(\iota(S_\phi(m)))$ is also a perfect graph. Hence, the result follows.
\end{proof}

\begin{corollary}
The best upperbound for the chromatic and clique number is given by $\chi(G_d(\iota(S_\phi(m))))=\omega(G_d(\iota(S_\phi(m)))) \leq \lfloor log_2(2^n -1)\rfloor +1$. 
\end{corollary}
\begin{proof}
Since for any $m\in \N$ there exists a smallest prime number $p$ such that $S_\phi(m)\subseteq S_\phi(p)$, $\phi((p)= n$. Therefore the result.
\end{proof}

In \cite{VG1}, we can find an easy algorithm to find the prime factors of any positive non-prime integer $n\geq 2$, which is as explained below: 

\textbf{Vishwas Garg Algorithm} \cite{VG1}. In adapted form it is informally described as follows:

\textit{Step 1:} If $n\geq 2$ is non-prime and odd let $o=n$, $\ell=0$ and let the set, $\mathcal P =\emptyset$ and go to Step 2. Else, divide by 2 iteratively, say $\ell$ iterations until an odd number say, $o$ is obtain. Let $\mathcal P \leftarrow \mathcal P \cup \{2\}$. If $o=1$,  go to Step 3 else, go to Step 2.

\textit{Step 2:} For $o$ and $i\in \{3,5,7,\ldots,max\{p: p\leq \sqrt o, p~an~odd~number\}\}$ and similarly as for 2 in Step 1, sequentially, divide iteratively by $i$ until 1 is obtained. Add (or $\cup$) all values of $i$ for which were divisors to the set $\mathcal P$. Go to Step 3.

\textit{Step 3:} Let $\mathcal P' = \mathcal P \cup \{t\cdot i: t\in \N~and~ t\cdot i\leq n\}$. Go to Step 4.

\textit{Step 4:} Let $S_\phi(n) = \{i:1\leq i\leq n-1\} -\mathcal P'$ and $\phi(n)=|S_\phi(n)|$.

Besides efficiency considerations, the above adapted algorithm provides the claimed results. Thus, the prime factors of $n$ are the elements of the set $\mathcal P$. This follows from the fact that the prime factors of an integer are unique. The division operator applied to a finite integer does exhaust in finite iterations. Hence, convergence of the algorithm follows. Also, since for any two consecutive prime numbers $p_1,p_2$ there exists a third prime number $p_3$ such that $2p_2<p_3<p_1p_2$, the iterative upper limit given by $\max\{p: p\leq \sqrt o, \text{$p$ an odd number}\}$, suffices. Steps 3 and 4 are well-defined, exhaustive and presents a unique result. The applicability of the adapted Vishwas Garg Algorithm is self evident in that following Step 4 it is easy to apply Definition \ref{Defn-3.1} and to obtain $\iota$-weights and to then construct a divisor graph. Therefore, if $\mathcal{P} = P_s(k)$ for some $k\in \N$, the result of Theorem \ref{2.3} applies to the corresponding divisor graph. In that sense the adapted Vishwas Garg Algorithm serves as methodology to test for $n\in \N$ whether or not $G_p(S_\phi(n))$ is acyclic without constructing the graph.

We denote the Euler Phi $\lcm$-relatively prime set-graph by, $G_p(\iota(S_\phi(n)))$. Then, 

\begin{lemma}
An Euler Phi $\lcm$-relatively prime set-graph $G_p(\iota(S_\phi(n)))$ has the maximum degree $\Delta (G_p(\iota(S_\phi(n))))= 2^{\phi(n))}-2$.
\end{lemma}
\begin{proof}
It follows from the fact that $\iota(v_{1,1}) = 1$ and $1$ is relatively prime to all integers hence, vertex $v_{1,1}(1)$ is adjacent to all vertices.
\end{proof}

We now present a relation between an Euler Phi $\lcm$-relatively prime set-graph and the corresponding Euler Phi $\lcm$-divisor set-graph.

\begin{theorem}
Let $\iota(S'_\phi(n))= \iota(S_\phi(n))-\{v_{1,1}(1)\}$. If for all pairs of distinct vertices $v_{s,i}(\iota),v_{t,j}(\iota) \in \iota(S'(\phi))$ the respective $\iota$-weights are either divisible or relatively prime, then $G_p(\iota(S_\phi(n)))= \overline{G}_d(\iota(S_\phi(n)-\{v_{1,1}(1)\})) + v_{1,1}(1)$. In other words, $G_d(\iota(S_\phi(n)))= \overline{G}_p(\iota(S_\phi(n)-\{v_{1,1}(1)\})) + v_{1,1}(1)$.
\end{theorem}
\begin{proof} The result follows from the fact that if for $v_{s,i}(\iota),v_{t,j}(\iota) \in \iota(S'(\phi))$ the respective $\iota$-weights are divisible, then $v_{s,i}(\iota),v_{t,j}(\iota)$ are not relatively prime and vice versa. Also $\iota(v_{1,1}(1))$ both divides and is relatively prime to all $\iota(v_{s,i}(\iota))$.  
\end{proof}

\begin{theorem}
For $n\in \N$ and $\phi(n)\geq 4$, both $G_p(\iota(S_\phi(n)))$ and $G_d(\iota(S_\phi(n)))$ contains at least one triangle.
\end{theorem}
\begin{proof}
For $n=1,2$ both $\phi(n)=1$ hence, $G_p(\iota(S_\phi(n)))$ and $G_d(\iota(S_\phi(n)))$ are an isolated vertex so are excluded. For $n=3$, $S_\phi(3)=\{1,2\}$ and hence, $G_d(\iota(S_\phi(3)))$ is a triangle and $G_p(\iota(S_\phi(3)))$ is path $P_3$ so are excluded. Similarly,  $G_d(\iota(S_\phi(4)))$, $G_p(\iota(S_\phi(4)))$, $G_d(\iota(S_\phi(6)))$ and $G_p(\iota(S_\phi(6)))$ are excluded.

For $n\in \N$ and $\phi(n)\geq 4$, let $S_\phi(n)= \{1, a_1,a_2,a_3,\ldots,a_\ell\}$. The vertices $v_{1,1}(1)$, $v_{1,a_i}(a_i)$ and $v_{2,(a_i-1)}(a_1)$ induce a triangle in $G_d(\iota(S_\phi(n)))$. Hence, $G_d(\iota(S_\phi(n)))$ contains at least one triangle.

If $n\geq 5$ is odd then, $2\in S_\phi(n)$. Furthermore, for such $n$ at least one prime number $2<p_1\leq n$, $p_1\in S_\phi(n)$, exists. Hence, vertices $v_{1,1}(1)$, $v_{1,2}(2)$ and $v_{1,p_1}(p_1)$ induce a triangle in $G_p(\iota(S_\phi(n)))$.

If $n\geq 8$ is even and since $\phi(n)\geq 4$ at least two prime numbers, $p_1,p_2 \in S_\phi(n)$, exist. Hence, vertices $v_{1,1}(1)$, $v_{1,p_1}(p_1)$ and $v_{1,p_2}(p_2)$ induce a triangle.

Therefore, for $n\geq 5~and~\neq 6$, $G_p(\iota(S_\phi(n)))$ contains at least one triangle.
\end{proof}


\section{Conclusion}

It was observed that the divisor graph, $G_d(S_\phi(n))= \overline{G}_p(S_\phi(n)-\{v_1\}) + v_1$ and the relatively prime graph, $G_p(S_\phi(n))= \overline{G}_d(S_\phi(n)-\{v_1\}) + v_1$. These observations permit the notion of a quasi-complement of any graph $G$. Let $G' =\langle V(G)-\{u\}_{u\in V(G)}\rangle$. A quasi-complement of $G$ is defined to be, $\overline{G}_u = \overline{G'}+u$. Hence, a finite family of quasi-complement exist. To be exact, $\nu(G)$ such quasi-complements. It easily follows that if $G$ is not connected and has an isolated vertex $u$ then, $\overline{G}=\overline{G}_u$. The notion of quasi-complements offers a new direction of research.

The study of the number of edges as well as the chromatic number of the derivative Euler Phi set-graphs ($\lcm$-divisor and $\lcm$-relatively prime) remains open.

It is known that integers $a,b$ can be: (i) not divisors and not relatively prime or, (ii) not divisors and relatively prime or, (iii) divisor and not relatively prime and finally, (iv) for say $a= 1$ which is both divisor and relatively prime to all positive integers. Therefore, a study of the relative divisor graph that is, vertex $v_a,v_b$ are adjacent if and only if $\gcd(a,b)\neq 1$ remains open. The authors are not aware of any studies related to this derivative graph. It will be advisable to exclude the element $1$, from the set under consideration. If this new graph is denoted by, $G_{rd}(\mathcal{S})$, $1\notin \mathcal{S}$, then clearly $G_d(\mathcal{S})\subseteq G_{rd}(\mathcal{S})$.




\end{document}